\documentclass{article}
\usepackage{graphicx}
\usepackage{amsfonts,mathtools}
\usepackage{subfigure}
\usepackage{epsfig}
\usepackage{cancel}
\usepackage{multirow}
\usepackage[usenames]{color}
\usepackage{xcolor}
\usepackage{array}
\usepackage{enumerate}

\newcommand\eqdef{\stackrel{\mathrm{def}}{=}}

\newcommand{\spp}{\mathop{\mathrm{span}}}

\newtheorem{proposition}{\bf Proposition}[section]
\newtheorem{theorem}{\bf Theorem}[section]
\newtheorem{proof}{\bf Proof}[section]

\newcommand*{\qed}{\hfill\ensuremath{\square}}%

\graphicspath{{figures/introduction/eps/}
              {figures/advection/eps/}
              {figures/burgers/eps/}
             }

\begin{document}

\title{Regularization of singularities in the weighted summation of Dirac-delta functions for the spectral
solution of hyperbolic conservation laws\thanks{This research was supported by the Air Force Office of Scientific Research (AFOSR-F9550-09-1-0097),
National Science Foundation (NSF-DMS-1115705) and the Computational Science Reseach Center (CSRC) at San Diego State University.
}
}


\author{
Jean-Piero Suarez and Gustaaf B. Jacobs\\
Aerospace Engineering and Computational Science Research Center,\\
 San Diego State University, San Diego, CA\\
gjacobs@mail.sdsu.edu
}


\maketitle

\begin{abstract}
Singular source terms
expressed as weighted summations of Dirac-delta functions are regularized
through approximation theory with convolution operators.
We consider the numerical solution of scalar and one-dimensional hyperbolic
conservation laws with the singular source by spectral Chebyshev collocation methods.
The regularization
is obtained by convolution with a high-order compactly supported Dirac-delta approximation whose overall accuracy
is controlled by the number of vanishing moments,
degree of smoothness and length of the support (scaling parameter).
An optimal scaling parameter that leads to a high-order accurate representation of the singular source at smooth parts
and full convergence order away from the singularities in the spectral solution is derived.
The accuracy of the
regularization and the spectral solution is assessed by solving an advection and Burgers equation with smooth initial data.
Numerical results illustrate the enhanced accuracy of the spectral method through the proposed regularization.
\end{abstract}

\section{Introduction}
This paper continues the work started in \cite{jp_siam_1}, on the regularization of singular source terms in the numerical solution
of hyperbolic conservation laws using spectral methods. We focus our attention on the scalar partial differential equations (PDEs)
\begin{equation}\label{eq_scalar_conser_law}
\dfrac{ \partial }{ \partial t  }Q(x,t) +  \dfrac{ \partial }{ \partial x }F(Q) = \underbrace{\sum_{i = 0}^{N_{p}} \dfrac{S(\xi_{i})}{n(\xi_{i})} K_{\varepsilon}(x - \xi_{i})}_{S_{\varepsilon}(x)},
\end{equation}
where $(x,t) \in \Omega \times [0,\infty)$ for some compact set $\Omega \subset \mathbb{R}$, $Q$ represents a conserved quantity,
$F$ is the flux function (smooth), $S \in L^{1}(\mathbb{R})$, $\{\xi_{i}\}_{i=0}^{N_{p}} \subset \Omega$ is a set of points where $S$
is known, $n(\xi_{i})$ is the number density for $\xi_{i}$, and the parametric family of functions $\{K_{\varepsilon}(\,\cdot\,)\;:\;\varepsilon > 0\} \subset L^{1}(\mathbb{R})$ is an approximation to the singular Dirac-delta distribution as the scaling parameter $\varepsilon$ tends to zero.
The weighted summation of Dirac-delta functions in the right hand side of (\ref{eq_scalar_conser_law}) represents a reconstruction of the function $S$ from the discrete set $\{\xi_{i}\}_{i=0}^{N_{p}}$.
Such reconstruction is a fundamental operation in many applications, including image and signal processing \cite{guus_friend,image_interpolation_5,image_interpolation_1,image_interpolation_4,image_interpolation_3,image_interpolation_2}, and particle methods
\cite{Abe,eastwood,jacobs,singular_source_1,sph_monaghan_2005}.

Our interest in (\ref{eq_scalar_conser_law}) is motivated by the numerical simulation of particle-laden flows with shocks using the
particle-source-in-cell (PSIC) method, introduced in \cite{psic}. In this framework, $\{\xi_{i}\}_{i=0}^{N_{p}}$
denotes the position of the $i$th particle, and $S$ is a weight function describing the influence of each particle onto the carrier
flow.

The full analysis of fluid particle interaction at high speeds involves the computation of the complete flow over each particle,
the tracking of individual solid or liquid complex particle boundaries along their paths, and the tracking
of shock waves in the moving frame. These individual computational components are difficult to resolve and are
currently barely within reach, even with the latest advances of computational technologies.

The PSIC method facilitates affordable computations of real geometries while accurately
representing individual particle dynamics. The main components of the PSIC method for computations of shocked particle-laden
flows are the following:
\begin{enumerate}[(i)]
\item Solution of hyperbolic conservation laws governing the carrier flow (Eulerian frame).
\item Integration of ordinary differential equations for particle's position and velocity (Lagrangian frame).
\item Interpolation of the flow properties at the particle's position.
\item Evaluation of a singular source term through weighted summation of Dirac-delta functions to account particle's effect on the carrier flow.
\end{enumerate}

High-order schemes for computations of (i), (ii), and (iii),
are available in the literature. For example, Jacobs and Don \cite{jacobs} developed and assessed a high-order PSIC algorithm for the computation of
the interaction between shocks, small scale structures, and liquid and/or solid particles in high-speed engineering applications. There,
computations of (i), (ii), and (iii) are performed by using the improved weighted essentially non-oscillatory (WENO)  scheme
\cite{don_wenoZ_2}, third order total variation diminishing (TVD) Runge–Kutta scheme \cite{Shu}, and essentially non-oscillatory interpolation  \cite{shu_eno_weno}, respectively. However,
(iv) is done by approximating the singular Dirac-delta with cardinal B-splines \cite{Abe}, which are second order accurate at best, on uniform grids.

The high-order accurate representation of the singular source term on general grids has not been fully accomplished \cite{jp_siam_1}, limiting
the accuracy of PSIC methods.  The presence of the Dirac-delta in (\ref{eq_scalar_conser_law}) leads to sharp particle interface
discontinuities in the source term. This is illustrated in Figure \ref{fig_single_and_multiple_particles}, for a distribution
of clustered particles and an isolated particle. The discontinuous interfaces may yield to nonphysical oscillations in the
solution or low order of accuracy away of the singularities, even when considering only one particle \cite{jp_siam_1}.

\begin{figure}[h]
\centering
\mbox{ \subfigure{ \includegraphics[width=5.80cm]{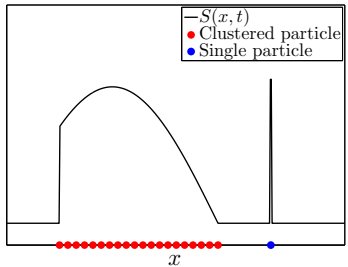} } }
\caption{One-dimensional representation of particle interface discontinuities in the singular source term $S_{\varepsilon}$ in (\ref{eq_scalar_conser_law}).}
\label{fig_single_and_multiple_particles}
\end{figure}

As is well-known, the nonlinearity in the flux function and the absence of dissipation terms with regularizing effect in (\ref{eq_scalar_conser_law}) may
lead to discontinuous solutions within finite time, even for smooth initial condition. In addition, the source term $S_{\varepsilon}$ can produce
discontinuities in the solution or any of its derivatives for $\varepsilon$ arbitrarily small, regardless of the nonlinearity of the flux
function and also the smoothness of the initial condition. When using  spectral methods to find the approximate solution
of (\ref{eq_scalar_conser_law}), it can hence be expected that Gibbs type phenomena cause loss of accuracy and numerical instability.

Under the assumption of a dense (large $N_{p}$) distribution of particles, i.e., sufficient number of particles
per computational cell, $S_{\varepsilon}$ can be regarded as a numerical approximation by quadrature of the convolution integral
\begin{equation}\label{eq_S_epsi_convolution}
\displaystyle (S \ast K_{\varepsilon})(x) \eqdef \int\limits_{\mathbb{R}} S(\xi)K_{\varepsilon}(x-\xi)\,\mathrm{d}\xi.
\end{equation}
Unfortunately, such ideal particle distribution can be destroyed by the influence of the carrier flow, scattering the
particles and creating a sparse distribution.

The approximation of Dirac-delta plays an important role in the accuracy of the spectral
solution to (\ref{eq_scalar_conser_law}) and the numerical evaluation of the source term. Top-hat and
piecewise linear functions are widely used to represent $K_{\varepsilon}$; however, their lack of smoothness produces
Gibbs oscillations in the spectral solution and inaccuracies in the quadratures rules to approximate
(\ref{eq_S_epsi_convolution}). High-order accurate and smooth representations of $K_{\varepsilon}$ on uniform grids have been proposed
in \cite{Monaghan,Monaghan_2,ryan_siac,shu_2013}, which are based on B-Splines and Gaussian functions. These regularization
techniques disrupt the convergence rate of spectral methods in the solution of hyperbolic conservation laws with
only one moving particle \cite{jp_siam_1}.

In \cite{jp_siam_1}, the precursor of this paper, we developed a regularization technique with optimal scaling. This approach is
based on a class of high-order compactly supported piecewise polynomials introduced in \cite{Tornberg}, to regularize a single
time-dependent  Dirac-delta source in spectral
approximations of one-dimensional hyperbolic conservation laws. The piecewise polynomials provide a high-order approximation to the Dirac-delta whose
overall accuracy is controlled by two conditions: number of vanishing moments and smoothness. Our proposed optimal scaling has shown to lead
to optimal order of convergence in the spectral solution of scalar (linear and nonlinear) problems and it has been successfully
implemented to solve the nonlinear system of Euler equations governing the compressible fluid dynamics with shocks and particles,
using the multidomain hybrid spectral-WENO scheme in \cite{spectralWENO}.

In this work, we develop a high-order accurate regularization technique with optimal scaling to approximate the source term
$S_{\varepsilon}(x)$  for a large number of clustered stationary particles. We seek for high-order accurate
solutions with spectral methods to (\ref{eq_scalar_conser_law}). The regularization is based on
the approximation of the convolution integral (\ref{eq_S_epsi_convolution}) by quadrature rules, substituting $K_{\varepsilon}$ by the
high-order regularization in \cite{jp_siam_1}. We derive an optimal scaling that ensures high-order convergence to $S$ at smooth parts,
according to the number of vanishing moments and the smoothness of $K_{\varepsilon}$.

The paper is structured as follows. Starting with the approximate Dirac-delta introduced
in \cite{Tornberg}, Section \ref{section_regularization} shows how a smooth and high-order
accurate regularization with optimal scaling that converges to (\ref{eq_S_epsi_convolution})
can be constructed. In Sections \ref{section_singular_advection_equation} and \ref{section_singular_burgers_equation} we assess the
accuracy of the spectral Chebyshev collocation method in the solution of a singular advection and Burgers equation with smooth
initial data, regularizing the respective source term with the optimal scaling. Finally, conclusions and direction for future work are
presented in Section \ref{conslusions}.

\section{Regularization of singular sources}\label{section_regularization}
Let $m \geq 1$ and $k \geq 0$ be integers, and let $P ^ {m,k}:[-1,1] \rightarrow \mathbb{R}$ be the polynomial
uniquely determined by the following conditions:
\begin{eqnarray}
\int\limits_{-1}^{1} P ^ {m,k} (\xi)\,\mathrm{d}\xi     & = & 1,                             \label{mass_conservation} \\
\bigl(P^{m,k}\bigr)^{(i)}(\pm 1)                        & = & 0 \mbox{ for } i = 0,\ldots,k, \label{smoothness}        \\
\int\limits_{-1}^{1} \xi^{i}P^{m,k}(\xi)\,\mathrm{d}\xi & = & 0 \mbox{ for } i = 1,\ldots,m. \label{moments}
\end{eqnarray}
Then, $P^{m,k}$ is a $2\left(\lfloor \frac{m}{2} \rfloor + k + 1\right)$th degree polynomial, containing only even powers of $\xi$, which can
be written as
\begin{equation}\label{eq_P_m_k}
P^{m,k}(\xi) = cw_{k}(\xi)Q(\xi),\;\; Q(\xi) = 1 - \sum_{j=1}^{\lfloor m/2 \rfloor} \langle 1, r_{2j} \rangle_{w_{k}} \, r_{2j}(\xi),
\end{equation}
where $\langle \cdot , \cdot \rangle_{w_{k}}$ denotes the usual weighted inner product on $C^{0}([-1,1])$ with positive weight function
$w_{k}(\xi) = (1 - \xi^{2})^{k+1}$, $\{r_{i}\}_{i=1}^{m}$ is an orthonormal Gram-Schmidt basis for $\spp\{\xi^{i}\}_{i = 1}^{m}$ with respect
to $\langle \cdot , \cdot \rangle_{w_{k}}$ and $c = 1 / \langle Q , 1 \rangle_{w_{k}}$ \cite{Tornberg}.

The polynomial $(\ref{eq_P_m_k})$ generates a class of compactly-supported piecewise polynomials that converge to the Dirac-delta in
the distributional sense as $\varepsilon \rightarrow 0^{+}$, given by \cite{Tornberg}
\begin{equation}\label{eq_dirac_delta_m_k}
\delta^{m,k}_{\varepsilon}(x) \eqdef
\left\{
\begin{array}{cc}
  \dfrac{1}{\varepsilon}P^{m,k}\left(\dfrac{x}{\varepsilon}\right), & |x| \leq \varepsilon, \\
  0, & |x| > \varepsilon.
\end{array}
\right.
\end{equation}
The conditions above on $P^{m,k}$ determine the convergence and the smoothness of (\ref{eq_dirac_delta_m_k}).
Specifically, condition (\ref{mass_conservation}) plus the compactness of the support are sufficient
to guarantee the convergence. Condition (\ref{smoothness}) ensures $C^{k}(\mathbb{R})$ smoothness.
Condition (\ref{moments}) is known as the moment condition. It determines the order of convergence in the approximation,
i.e. $\mathcal{O}(\varepsilon^{m+1})$ for $m$ vanishing moments.

Hereinafter, we shall show that (\ref{eq_dirac_delta_m_k}) provides arbitrary accuracy with any desired
smoothness to regularize Lebesgue integrable functions that are smooth on compact subset of $\mathbb{R}$.
The operation of convolution provides the tool to build a high-order accurate and smooth approximations.

Concretely, given $S \in L^{1}(\mathbb{R})$, let $S_{\varepsilon}^{m,k}$ be the function defined by the convolution
\begin{equation}\label{eq_S_m_k}
S_{\varepsilon}^{m,k}(x) \eqdef (S \ast \delta^{m,k}_{\varepsilon})(x) = \int\limits_{\mathbb{R}} S(\xi)\delta^{m,k}_{\varepsilon}(x-\xi)\,\mathrm{d}\xi = \int\limits_{x - \varepsilon}^{x + \varepsilon} S(\xi)\delta^{m,k}_{\varepsilon}(x-\xi)\,\mathrm{d}\xi.
\end{equation}
Then, (\ref{eq_S_m_k}) is a $C^{k}(\mathbb{R})$ class function that converges to $S$ in the $L^{1}(\mathbb{R})$ sense. This follows from
Proposition \ref{prop_smoothness_S_m_k} and Theorem \ref{thm_convergence_S_m_k}, a classical results of real analysis \cite{folland}.

\begin{proposition}\label{prop_smoothness_S_m_k}
Suppose that $f \in L^{1}(\mathbb{R})$, $g \in C^{k}(\mathbb{R})$ and the derivatives $g^{(i)}$ are bounded for $i = 0, \ldots, k$. Then
$f \ast g \in C^{k}(\mathbb{R})$ and $\bigl(f \ast g\bigr)^{(i)} = f \ast g^{(i)}$.
\end{proposition}

\begin{theorem}\label{thm_convergence_S_m_k}
Let $\varphi \in L^{1}(\mathbb{R})$ such that $\displaystyle \int\limits_{\mathbb{R}} \varphi (\xi)\,\mathrm{d}\xi = 1$ and let $\varphi_{\varepsilon}(x) \eqdef \dfrac{1}{\varepsilon}\varphi \left(\dfrac{x}{\varepsilon}\right)$.
If $f \in L^{1}(\mathbb{R})$, then $f \ast \varphi_{\varepsilon}$ converge to $f$ in the $L^{1}(\mathbb{R})$ norm as $\varepsilon \rightarrow 0^{+}$.
\end{theorem}

An arbitrary order of accuracy in the approximation of $S$ by (\ref{eq_S_m_k}) can be achieved locally, according to the number of
vanishing moments of (\ref{eq_dirac_delta_m_k}). We establish this result in Theorem \ref{thm_convergence_moments_S_m_k}.

\begin{theorem}\label{thm_convergence_moments_S_m_k}
Let $S \in L^{1}(\mathbb{R})$ and let $[a,b] \subset \mathbb{R}$ be a compact set. If $S \in C^{m}([a,b])$
and $S^{(m+1)}$ exists and is bounded on $(a,b)$, then $S_{\varepsilon}^{m,k}$ converges pointwise to $S$ on $(a,b)$ as $\varepsilon \rightarrow 0^{+}$. Moreover,
\begin{equation}\label{eq_error_S_m_k}
S_{\varepsilon}^{m,k}(x) - S(x) = \mathcal{O}({\varepsilon^{m+1}}) \mbox{ for } x \in (a + \varepsilon , b - \varepsilon).
\end{equation}
\end{theorem}

\begin{proof}
Let $\varepsilon > 0$ be sufficiently small such that $\varepsilon < \dfrac{b-a}{2}$. Expanding $S$ in a Taylor series
on $\xi \in [a,b]$, around $x \in (a + \varepsilon, b - \varepsilon)$, we obtain
\begin{eqnarray*}
S_{\varepsilon}^{m,k}(x) & = & \sum_{j = 0}^{m} \dfrac{S^{(j)}(x)}{j!} \int\limits_{a}^{b}(\xi - x)^{i} \delta^{m,k}_{\varepsilon}(\xi - x)\,\mathrm{d}\xi \\
                         & + & \dfrac{1}{(m+1)!}\int\limits_{a}^{b} (\xi - x)^{m+1}\delta^{m,k}_{\varepsilon}(\xi - x)S^{(m+1)}(\zeta)\,\mathrm{d}\xi,
\end{eqnarray*}
for some $\zeta$ between $\xi$ and $x$. Note that since $[-\varepsilon, \varepsilon] \subset (a - x, b - x)$,
\begin{equation*}
\int\limits_{a}^{b}(\xi - x)^{i} \delta^{m,k}_{\varepsilon}(\xi - x)\,\mathrm{d}\xi = \varepsilon^{i}\int\limits_{-1}^{1} \xi^{i} P^{m,k}(\xi)\,\mathrm{d}\xi = \left\{
\begin{array}{cl}
1,                              & i = 0,          \\
0,                              & i = 1,\ldots,m,
\end{array}
\right.
\end{equation*}
and therefore
\begin{equation*}
\left\lvert S(x) - S^{m,k}_{\varepsilon}(x) \right\rvert \leq \dfrac{\kappa\varepsilon^{m+1}}{(m+1)!}\int\limits_{-1}^{1} \lvert \xi \rvert ^{m+1} \lvert P^{m,k}(\xi) \rvert\,\mathrm{d}\xi,
\end{equation*}
for some constant $\kappa \in  \mathbb{R}_{+}$. \qed
\end{proof}

In computational implementations, the use of numerical integration to evaluate (\ref{eq_S_m_k}) is required. The main challenge
is to preserve the $(m+1)$th order of accuracy established in (\ref{eq_error_S_m_k}), when $S^{m,k}_{\varepsilon}$
is replaced by an approximation $\tilde{S}^{m,k}_{\varepsilon}$ computed with a quadrature formula. According to
Theorem \ref{thm_convergence_moments_S_m_k}, the error estimation
\begin{equation*}
\left\lvert S(x) - \tilde{S}^{m,k}_{\varepsilon}(x) \right\rvert \leq \underbrace{\left\lvert S(x) - S^{m,k}_{\varepsilon}(x) \right\rvert}_{ \mathcal{O}(\varepsilon^{m+1}) \mbox{ \small{by (\ref{eq_error_S_m_k})} }} + \left\lvert S^{m,k}_{\varepsilon}(x) - \tilde{S}^{m,k}_{\varepsilon}(x) \right\rvert,
\end{equation*}
shows that the desired order of accuracy can be reached by imposing the following condition on the quadrature error:
\begin{equation}\label{eq_error_tilde_S_m_k}
S^{m,k}_{\varepsilon}(x) - \tilde{S}^{m,k}_{\varepsilon}(x) = \mathcal{O}(\varepsilon^{m+1}) \mbox{ for } x \in (a + \varepsilon , b - \varepsilon).
\end{equation}

Generally, the quality of the approximation resulting from quadrature rules depends on
continuity, the number of continuous derivatives and their magnitude of the integrand. Observe that
\begin{equation*}
\dfrac{\partial^{i}\left(S(\xi)\delta^{m,k}_{\varepsilon}(x-\xi)\right)}{\partial \xi^{i}} = \dfrac{1}{\varepsilon^{i+1}} \dfrac{\partial^{i}\left(S(\xi)P^{m,k}\left(\dfrac{x-\xi}{\varepsilon}\right)\right)}{\partial\xi^{i}},\;
\end{equation*}
for any integer $i \geq 0$ and $\xi \in (x-\varepsilon,x+\varepsilon)$. Thus, an arbitrarily small choice of the scaling parameter $\varepsilon$ might affect the validity of (\ref{eq_error_tilde_S_m_k}).

In Proposition \ref{prop_smoothness_optimal_epsi}, we will introduce an optimal scaling parameter that preserves the $(m+1)$th order of accuracy in (\ref{eq_error_tilde_S_m_k}).
Without loss of generality, suppose that $\tilde{S}^{m,k}_{\varepsilon}$ is computed using the composite Newton-Cotes formulas (closed) \cite{rabinowitz} on $N_{p}$ subintervals
$\left\{\bigl[\xi_{i},\xi_{i+1}\bigr]\right\}_{i = 0}^{N_{p}-1}$ given by the points
\begin{equation*}
\xi_{0} = a < \xi_{1} < \cdots < \xi_{N_{p}} = b,
\end{equation*}
i.e.,
\begin{eqnarray}
S^{m,k}_{\varepsilon}(x)  =  \sum_{i = 0}^{N_{p} - 1} \int\limits_{\xi_{i}}^{\xi_{i+1}} S(\xi) \delta^{m,k}_{\varepsilon}(x - \xi)\,\mathrm{d}\xi   & = &  \underbrace{\sum_{i = 0}^{N_{p} - 1} \left( \sum_{j = 0}^{q} \alpha_{ij}S(\xi_{i} + j h_{i}) \delta^{m,k}_{\varepsilon}(x - \xi_{i} - j h_{i}) \right)}_{\tilde{S}^{m,k}_{\varepsilon}(x)}  \nonumber\\
                                                                                                                                                   & + & \kappa \sum_{i = 0}^{N_{p} - 1} h_{i}^{q+2} \dfrac{\partial^{q+1}\left(S(\xi)\delta^{m,k}_{\varepsilon}(x-\xi)\right)}{\partial\xi^{q+1}} \bigg|_{\xi = \zeta_{i}},                \label{eq_newton_cotes_tilde_S_m_k}
\end{eqnarray}
where $q$ is a nonnegative integer denoting the degree of exactness of the quadrature rule, $\alpha_{ij}$ are the quadrature weights, $h_{i} = \left(\xi_{i+1} - \xi_{i}\right)/q$, $\kappa$ is a constant independent
of the integrand and $h_{i}$, and $\zeta_{i} \in (\xi_{i},\xi_{i+1})$.

\begin{proposition}\label{prop_smoothness_optimal_epsi}
Under the hypotheses in Theorem \ref{thm_convergence_moments_S_m_k}, suppose that $m,k\geq 2$ and let $q \leq \min\{m,k\} - 1$
be the degree of exactness in the Newton-Cotes quadrature rule. Then,
\begin{equation*}
S(x) - \tilde{S}^{m,k}_{\varepsilon}(x) = \mathcal{O}(\varepsilon^{m+1}) \mbox{ for } x \in (a + \varepsilon , b - \varepsilon),
\end{equation*}
provided that the optimal scaling parameter $\displaystyle \varepsilon = \mathcal{O}\left(\left( \sum_{i=0}^{N_{p}-1} h_{i}^{q+2} \right)^{1/(m+q+3)}\right)$.
\end{proposition}

\begin{proof}
From (\ref{eq_newton_cotes_tilde_S_m_k}), it follows that
\begin{equation*}
S^{m,k}_{\varepsilon}(x) - \tilde{S}^{m,k}_{\varepsilon}(x) = \kappa\sum_{i=0}^{N_{p}-1} h_{i}^{q+2}\dfrac{\partial^{q+1}\left(S(\xi)\delta^{m,k}_{\varepsilon}(x-\xi)\right)}{\partial\xi^{q+1}} \bigg|_{\xi = \zeta_{i}}.
\end{equation*}
Using the Leibnitz rule, it can be proven that \cite{jp_siam_1}
\begin{equation*}
\dfrac{\partial^{q+1}\left(S(\xi)\delta^{m,k}_{\varepsilon}(x-\xi)\right)}{\partial\xi^{q+1}} \bigg|_{\xi = \zeta_{i}} = \mathcal{O}\left(\dfrac{1}{\varepsilon^{q+2}}\right) \mbox{ for } i = 0,\ldots,N_{p}-1;
\end{equation*}
and hence, (\ref{eq_error_tilde_S_m_k}) leads to the proposed optimal
scaling parameter. \qed
\end{proof}

\section{Numerical results and discussion}
In this section, we assess the accuracy of the spectral Chebyshev collocation method with Gauss--Lobatto nodes \cite{hesthave_sigal_david} in
the solution of a singular advection and inviscid Burgers equation, using  the high-order regularization technique with optimal scaling presented
in Section \ref{section_regularization}.

\subsection{Singular advection equation}\label{section_singular_advection_equation}
Consider the following first order advection PDE on the domain $(x,t) \in [-1,1] \times [0,2]$:
\begin{equation}\label{eq_original_advection}
 \left\{
 \begin{array}{rcl}
   \dfrac{\partial u}{\partial t} + \dfrac{\partial u}{\partial x} & = &  \underbrace{3\cos(5 \pi x)\left( H\left( x + \dfrac{3}{10}\right) - H\left( x - \dfrac{3}{10} \right) \right)}_{S(x)}  , \\
   u(x,0)                                                          & = &  \sin(\pi x),  \\
   u(-1,t)                                                         & = &  \sin(\pi(-1-t)),
 \end{array}
 \right.
 \end{equation}
where $H$ denotes the Heaviside function
\begin{equation*}\label{eq_heaviside_function}
H(x) = \left\{\begin{array}{cl}
         0, & x<0, \\
         \dfrac{1}{2}, & x=0, \\
         1, & x>0,
       \end{array}\right.
\end{equation*}
and $S \in C^{0}([-1,1])$ denotes the singular source term, whose first order derivative is discontinuous at $x = \pm \dfrac{3}{10}$. The
analytical solution to (\ref{eq_original_advection}) is the $C^{0}([-1,1])$ function
\begin{equation}\label{eq_exact_solution_original_advection}
u(x,t) = \sin\left( \pi (x - t) \right) + \int\limits_{x-t}^{x} S(\xi)\,\mathrm{d}\xi,
\end{equation}
which has discontinuous first order derivatives at $x = \pm \dfrac{3}{10} + t$.



In (\ref{eq_original_advection}), the source term leads to first order singularities
in the analytical solution (\ref{eq_exact_solution_original_advection}). When approximating (\ref{eq_original_advection})
with a spectral collocation method combined with a third order TVD Runge--Kutta scheme \cite{Shu},
spectral accuracy is prevented by Gibbs oscillations, as illustrated in Figure \ref{fig_U_spectral_multiple_particles_advection_S}, where
$u^{N}(x,2)$ denotes the spectral solution at $t = 2$ using $N+1$ Gauss--Lobatto nodes with the Courant--Friedrich--Lewy condition $\Delta t = 0.5/N^{2}$.

\begin{figure}[h]
\centering
\mbox{ \subfigure{ \includegraphics[width=5.80cm]{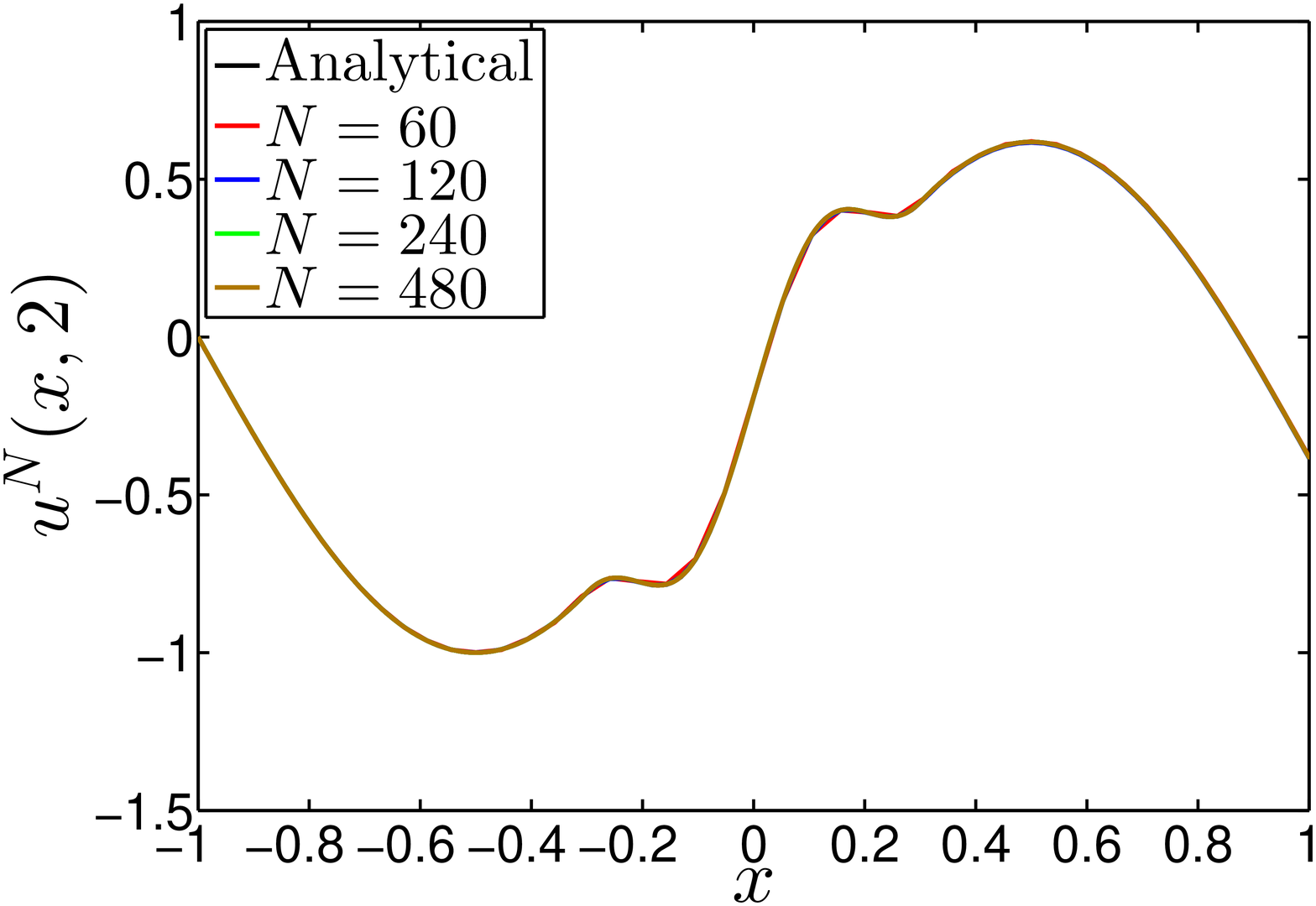} } \quad \subfigure{ \includegraphics[width=5.80cm]{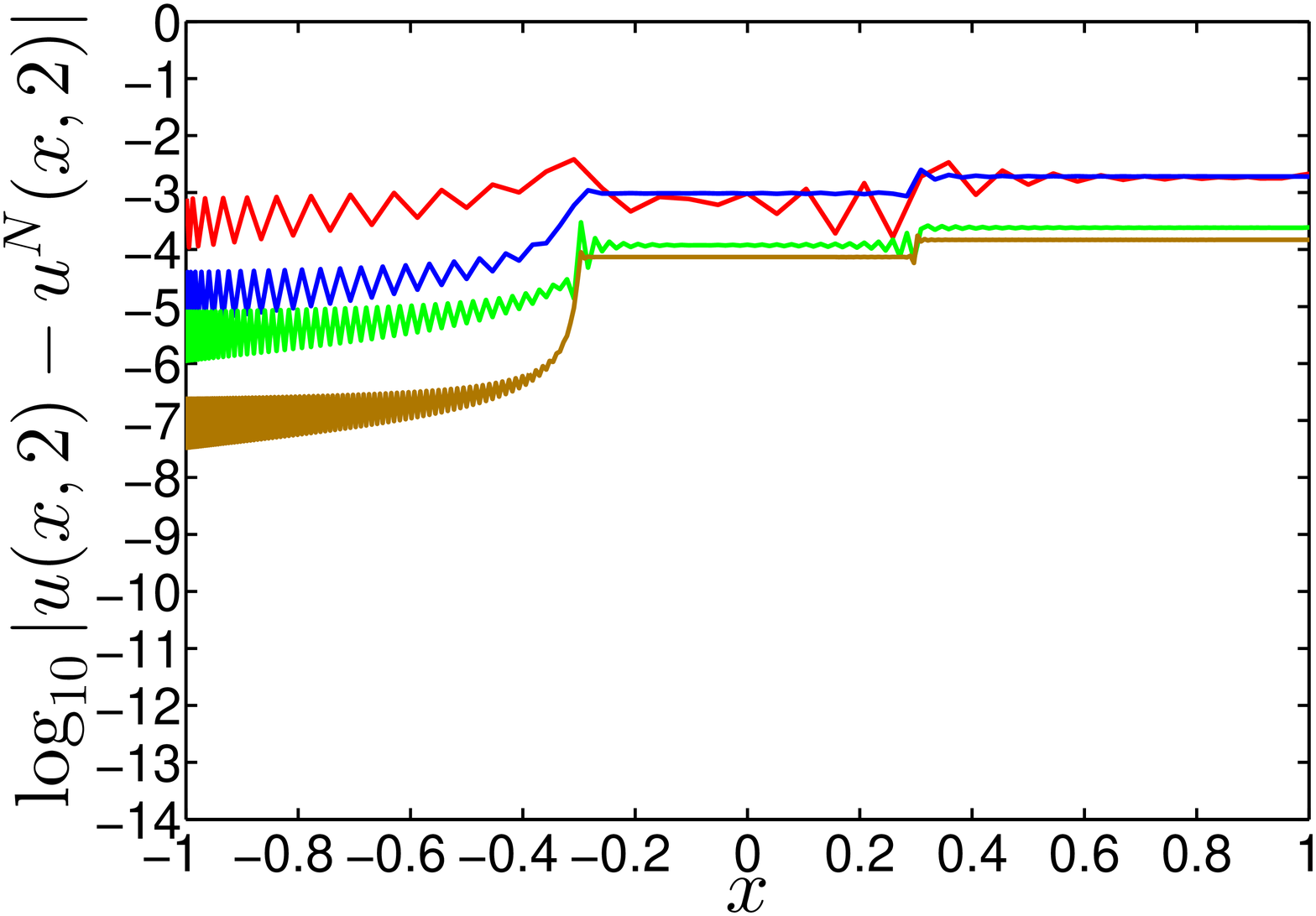} } }
\caption{Analytical solution, spectral solution (left) and pointwise error (right) to the advection equation at $t = 2$, without regularization of the singular source term.}
\label{fig_U_spectral_multiple_particles_advection_S}
\end{figure}

To suppress the Gibbs phenomenon, we shall henceforth approximate the singular source term $S$ by the smooth function
$\tilde{S}^{m,k}_{\varepsilon} \in C^{k}([-1,1])$ in (\ref{eq_newton_cotes_tilde_S_m_k}), using the composite
Simpson quadrature rule on $N_{p} + 1$ nonuniform points given by
\begin{equation*}
\xi_{i} = \dfrac{3}{10}\sin\left(\pi \left(-\dfrac{1}{2}+\dfrac{i}{N_{p}}\right)\right) \mbox{ for } i = 0,\ldots,N_{p},
\end{equation*}
and the optimal scaling parameter $\varepsilon$ in Proposition \ref{prop_smoothness_optimal_epsi}. Without loss of generality,
we consider $N_{p} = 2001$ (see Figure \ref{fig_source_U_spectral_multiple_particles_advection_S_7_4}).
The same accuracy can be obtained by using a more accurate quadrature rule with fewer points, for instance,
a five points Newton--Cotes quadrature rule with $N_{p} = 301$.

We assess accuracy within three separate regions of the spatial domain $\Omega \eqdef [-1,1]$, i.e.,
\begin{eqnarray}
\mathcal{P}_{\varepsilon} & \eqdef & \left( \xi_{0} + \varepsilon, \xi_{N_{p}} - \varepsilon \right),                                                                   \label{eq_zone_P} \\
\mathcal{R}_{\varepsilon} & \eqdef & \left[ \xi_{0} - \varepsilon , \xi_{0} + \varepsilon \right] \cup \left[ \xi_{N_{p}} - \varepsilon, \xi_{N_{p}} + \varepsilon \right],     \nonumber \\
\mathcal{Q}_{\varepsilon} & \eqdef & \Omega - \left( \mathcal{P}_{\varepsilon} \cup \mathcal{R}_{\varepsilon} \right) \label{eq_zone_Q}.
\end{eqnarray}
Specifically, $\mathcal{P}_{\varepsilon}$ contains the particles $\left\{\xi\right\}_{i=0}^{N_{p}}$. On $\mathcal{R}_{\varepsilon}$, the
singularities in $S$ are removed by using a high-order polynomial.  Within $\mathcal{Q}_{\varepsilon}$, the compactness of the support
in (\ref{eq_dirac_delta_m_k}) implies that $S$ and $\tilde{S}^{m,k}_{\varepsilon}$ are equal.

Figure \ref{fig_source_U_spectral_multiple_particles_advection_S_7_4} shows the singular and regularized source terms (left), as well as the respective pointwise error (right),
for the case of $m = 7$, $k = 4$ and optimal scaling $\varepsilon = 6.6 \times 10^{-2}$. By definition, the regularization does not introduce errors
on $\mathcal{Q}_{\varepsilon}$. On $\mathcal{R}_{\varepsilon}$, the regularization error is affected by the smoothing that removes the first order singularity in $S$, and therefore high-order
convergence is not expected. The number of vanishing moments $m$ governs the behavior of the  regularization error on $\mathcal{P}_{\varepsilon}$
and $\mathcal{O}\left(\varepsilon^{m+1}\right)$ convergence can be achieved for $\varepsilon$ sufficiently small
($\varepsilon < \frac{3}{10}$, according to the proof of Theorem \ref{thm_convergence_moments_S_m_k}).

\begin{figure}[h]
\centering
\mbox{ \subfigure{ \includegraphics[width=5.80cm]{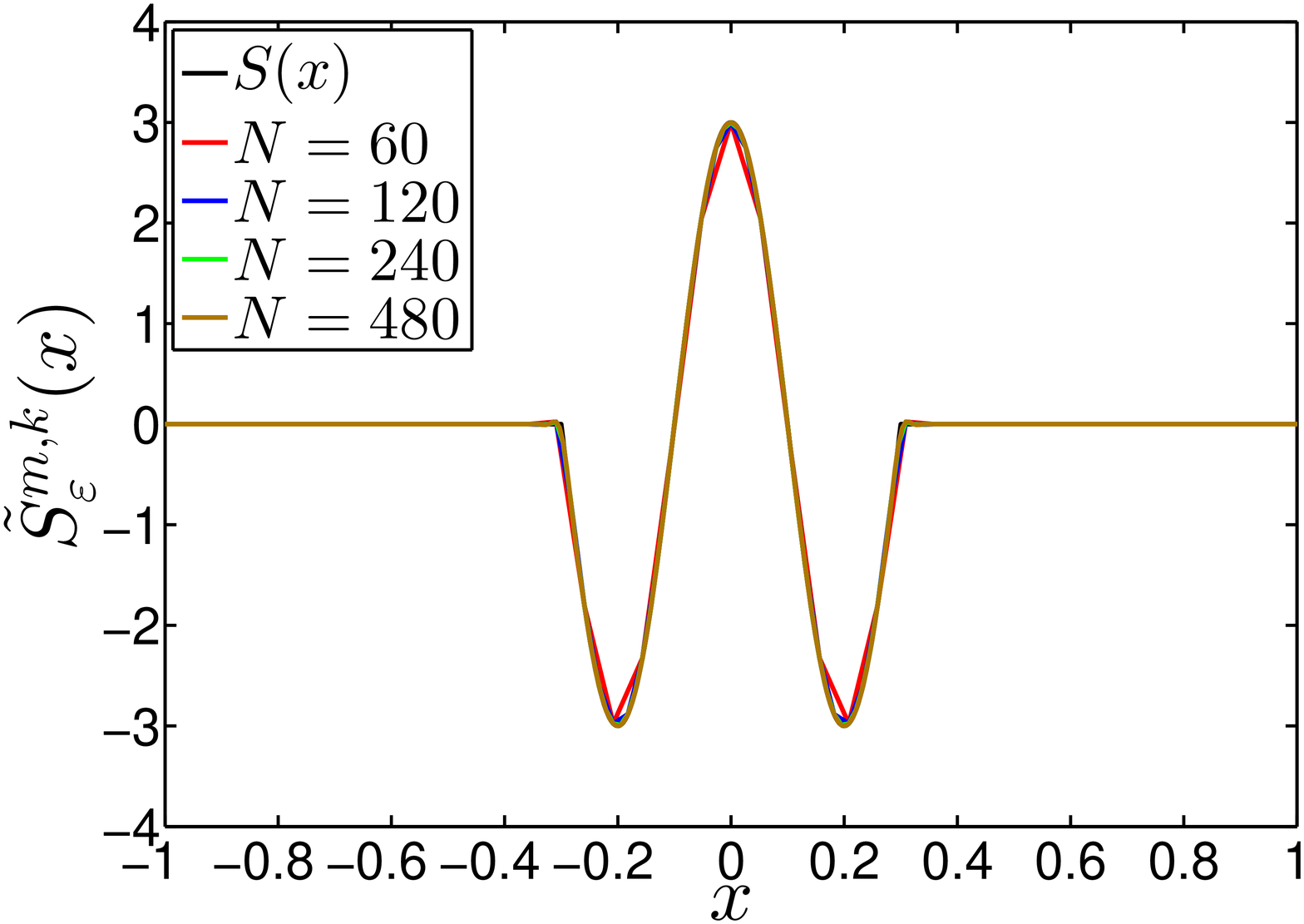} } \quad \subfigure{ \includegraphics[width=5.80cm]{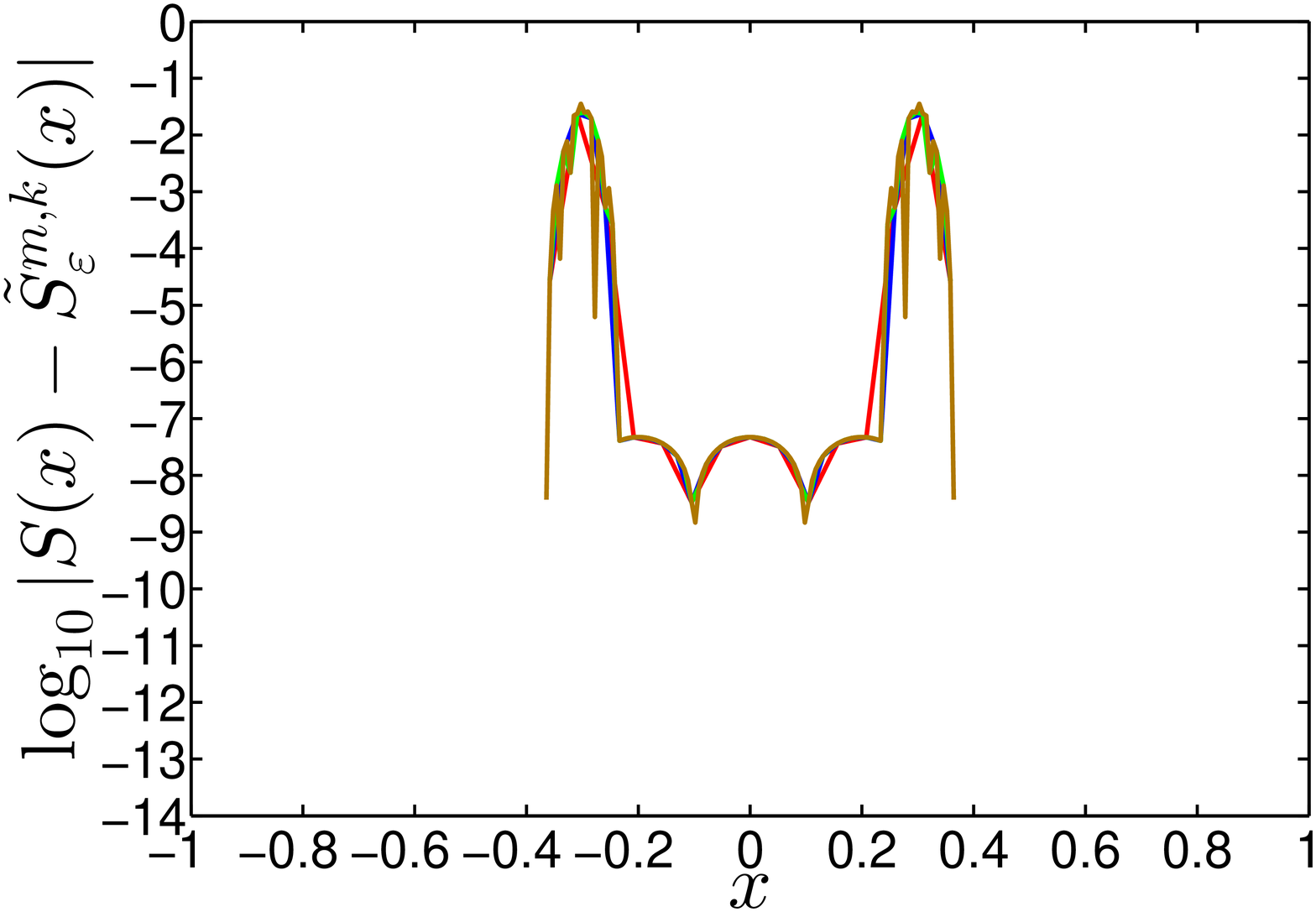} } }
\caption{Singular source, regularized source for $m = 7$, $k=4$ and optimal scaling $\varepsilon = 6.6 \times 10^{-2}$ (left), and pointwise error (right) on $N+1$ spectral points.}
\label{fig_source_U_spectral_multiple_particles_advection_S_7_4}
\end{figure}

Let us now study the accuracy of the spectral method applied to (\ref{eq_original_advection}) when $S$
is replaced by the regularization $\tilde{S}^{m,k}_{\varepsilon}$ shown in Figure \ref{fig_source_U_spectral_multiple_particles_advection_S_7_4} (left), compared with the analytical solution (\ref{eq_exact_solution_original_advection}).
The spectral solution computed with $N+1$ Gauss--Lobatto nodes using the regularized source term will be denoted by $u^{N}_{\varepsilon}(x,t)$.
Since $\tilde{S}^{m,k}_{\varepsilon}\in C^{k}([-1,1])$, we expect to have a convergence of $\mathcal{O}\left(N^{-(k+1)}\right)$ on $\mathcal{P}_{\varepsilon} \cup \mathcal{Q}_{\varepsilon}$
with the regularization of the singular source term. On $\mathcal{R}_{\varepsilon}$, the regularization error clearly will cause
a loss of accuracy.

In Figure \ref{fig_U_spectral_multiple_particles_advection_S_7_4}, the analytical and spectral solution (left) and the respective pointwise error (right) at $t = 2$ are
shown. The regularization improves the accuracy of the spectral solution on $\mathcal{P}_{\varepsilon} \cup \mathcal{Q}_{\varepsilon}$  with grid refinement, as opposed
to Figure \ref{fig_U_spectral_multiple_particles_advection_S}, where the spectral solution was computed with the singular source term. Figure \ref{fig_conv_error_U_spectral_multiple_particles_advection_S_7_4} shows the convergence order on $\mathcal{P}_{\varepsilon}$ (left) and $\mathcal{Q}_{\varepsilon}$
(right), which has been estimated using linear regression to fit the respective $L^{2}_{w}(-1,1)$ error ($w(x) = (1-x^{2})^{-1/2}$). The convergence
is faster than $\mathcal{O}\left(N^{-(k+1)}\right)$ (superconvergence), i.e., $\mathcal{O}\left(N^{-6.45}\right)$ and $\mathcal{O}\left(N^{-7.71}\right)$  on $\mathcal{P}_{\varepsilon}$ and $\mathcal{Q}_{\varepsilon}$,
respectively.

\begin{figure}[h]
\centering
\mbox{ \subfigure{ \includegraphics[width=5.80cm]{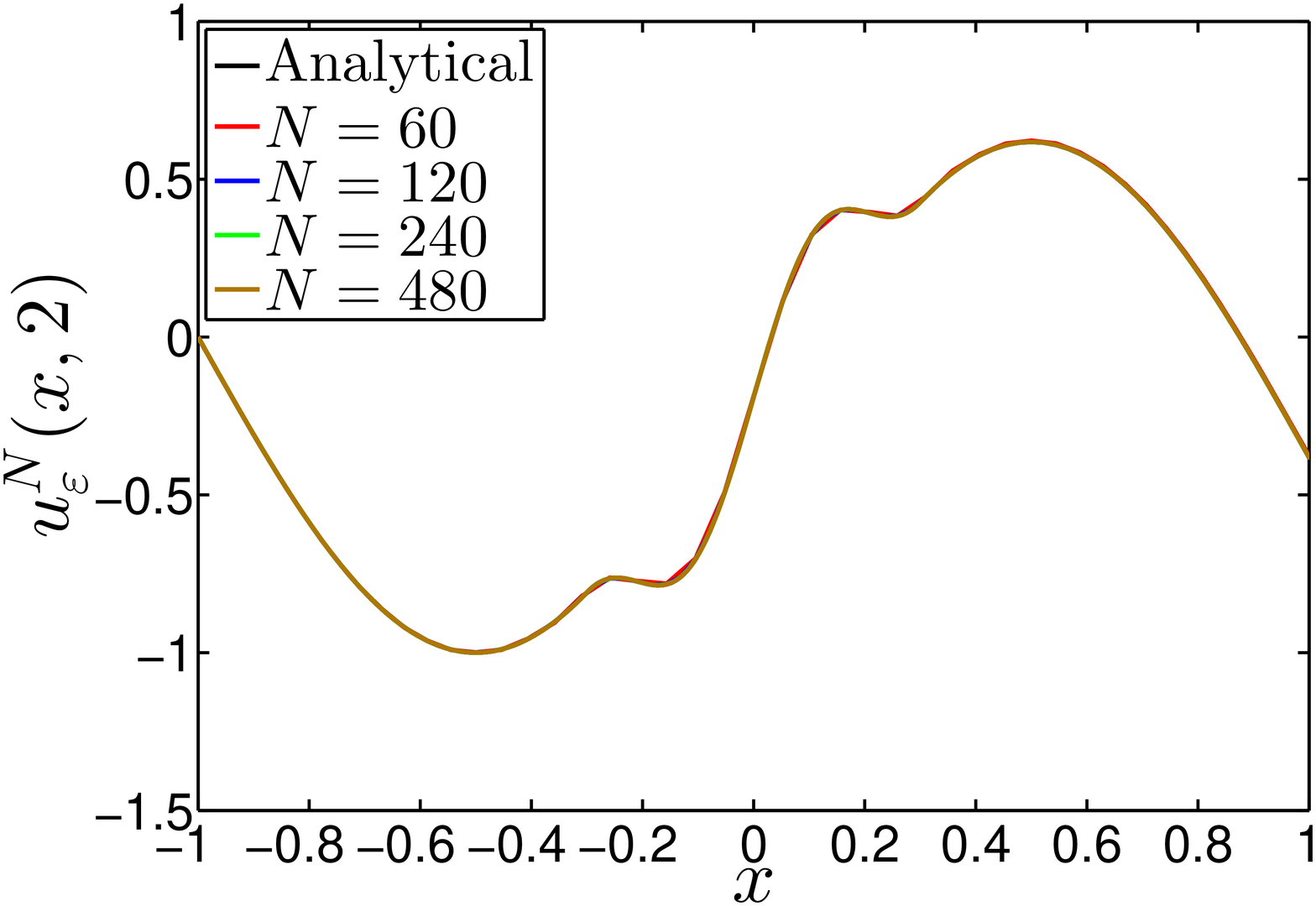} } \quad \subfigure{ \includegraphics[width=5.80cm]{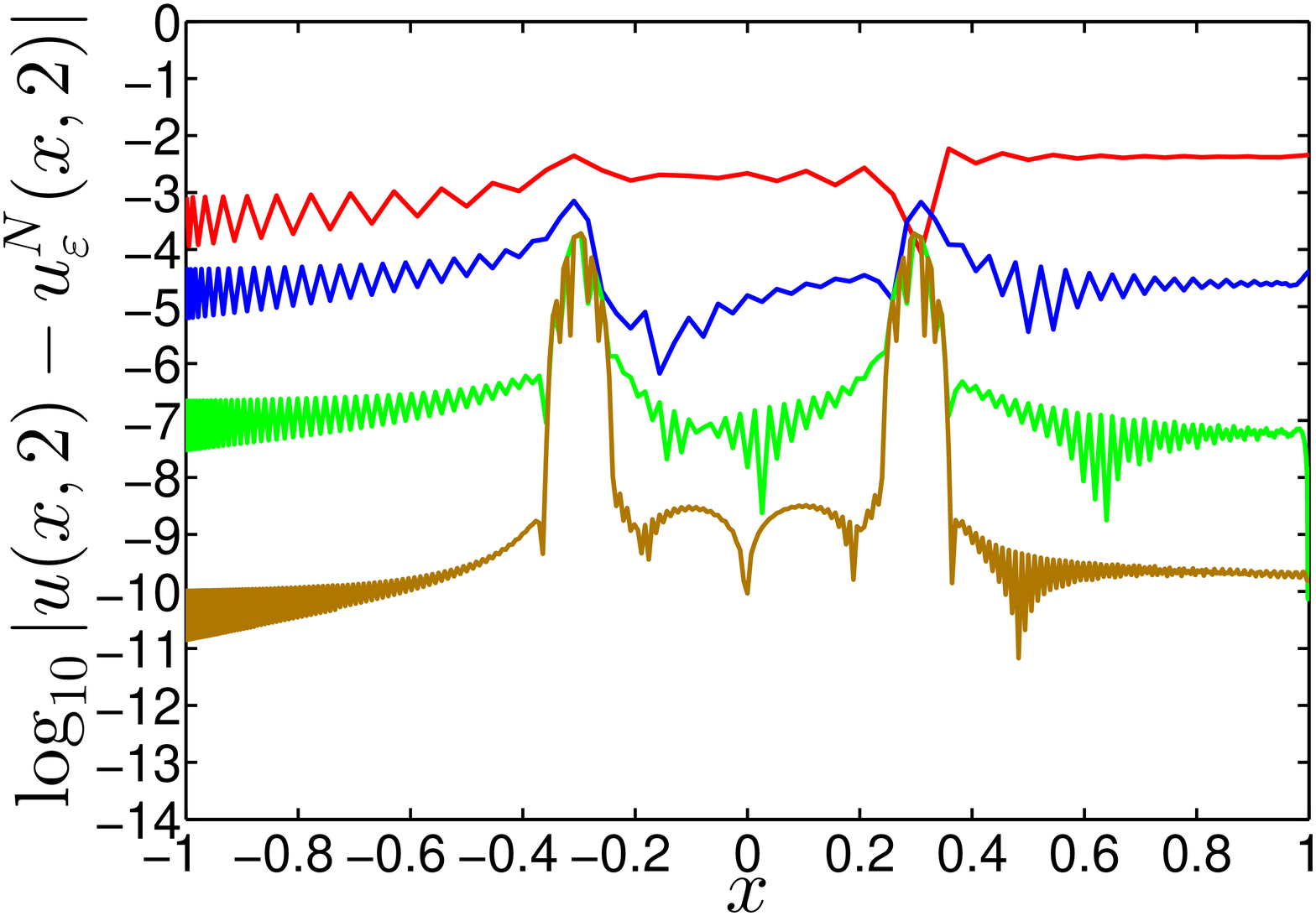} } }
\caption{Analytical solution, spectral solution (left) and pointwise error (right) to the advection equation at $t = 2$, using the regularization $S^{m,k}_{\varepsilon}$.}
\label{fig_U_spectral_multiple_particles_advection_S_7_4}
\end{figure}

\begin{figure}[h]
\centering
\mbox{ \subfigure{ \includegraphics[width=5.80cm]{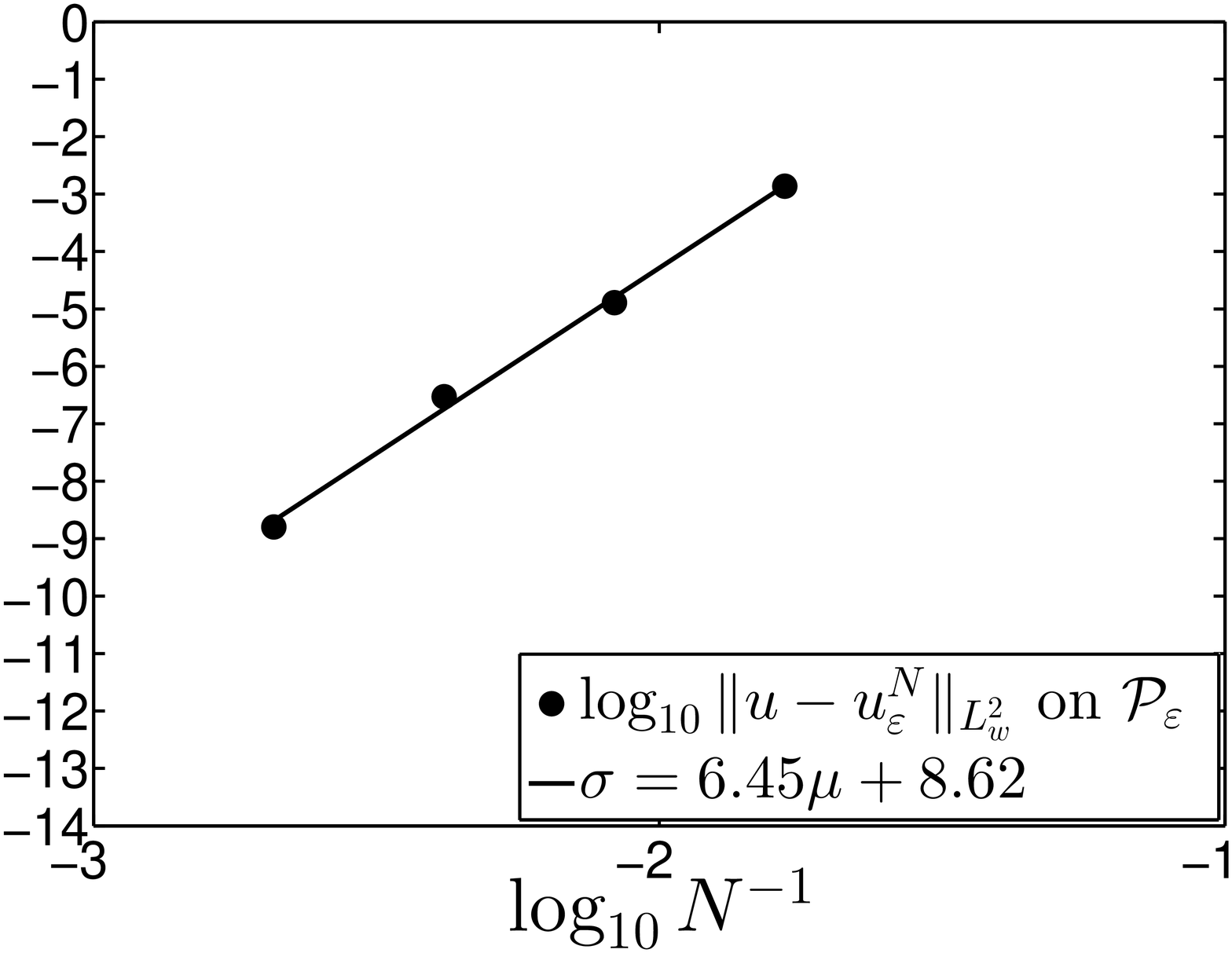} } \quad \subfigure{ \includegraphics[width=5.80cm]{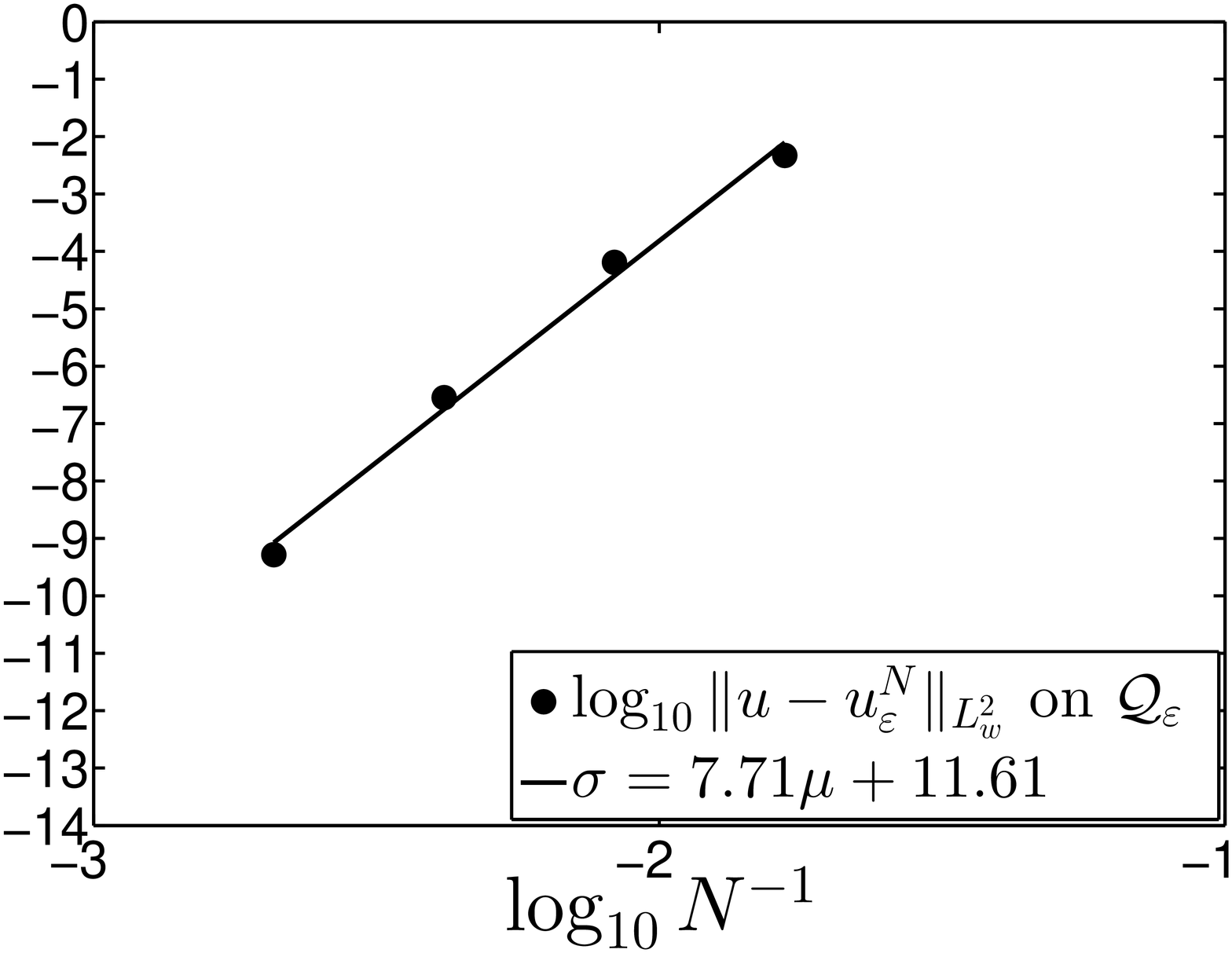} } }
\caption{Error in $\|\cdot\|_{L^{2}_{w}}$ norm and linear regression to estimate convergence order on $\mathcal{P}_{\varepsilon}$ (left) and $\mathcal{Q}_{\varepsilon}$ (right) to the advection equation at $t = 2$, using the regularization $S^{m,k}_{\varepsilon}$.}
\label{fig_conv_error_U_spectral_multiple_particles_advection_S_7_4}
\end{figure}

Fixing $k$, the accuracy of the spectral solution is controlled by $m$ and the optimal scaling $\varepsilon$. For $k = 4$,
the expected $\mathcal{O}\left(N^{-(k+1)}\right)$ convergence on $\mathcal{P}_{\varepsilon} \cup \mathcal{Q}_{\varepsilon}$
requires at least $m = 7$ vanishing moments. By increasing $m$, a more accurate representation of the singular source on $\mathcal{P}_{\varepsilon}$,
through the regularization,  is obtained. However, the optimal scaling $\varepsilon$ increases as long as
$m$ does it, so that an arbitrary large choice of $m$ can lead to a violation of the condition $\varepsilon < \frac{3}{10}$, i.e, an increase in the size of the
regularization zone $\mathcal{R}_{\varepsilon}$ where high-order of accuracy will not be reached. In terms of accuracy of the spectral solution,
superconvergence is observed for $m \geq 7$. Results for $m = 1,5,9,13,17$ and $k = 4$ are summarized in Table \ref{table_m_k_epsi_order_accuracy_advection}.

\begin{table}[!h]
\begin{center}
\caption{Optimal scaling, pointwise error in the regularization of the source term on $\mathcal{P}_{\varepsilon}$ and convergence order of the spectral solution to the advection equation at $t = 2$ on $\mathcal{P}_{\varepsilon} \cup \mathcal{Q}_{\varepsilon}$, when S is replaced by $\tilde{S}^{m,k}$.}
\label{table_m_k_epsi_order_accuracy_advection}
\def\arraystretch{2.0}
\footnotesize{
\begin{tabular}{|l|l|l|l|l|}
\hline
\multicolumn{5}{|c|}{$k = 4$} \\
\hline
\multicolumn{1}{|c|}{$m$} & \multicolumn{1}{c|}{$\varepsilon$} & \multicolumn{1}{c|}{$\log_{10} \lvert S(x) - \tilde{S}^{m,k}(x) \rvert$ on $\mathcal{P}_{\varepsilon}$} & \multicolumn{1}{c|}{Conv. order on $\mathcal{P}_{\varepsilon}$} & \multicolumn{1}{c|}{Conv. order on $\mathcal{Q}_{\varepsilon}$} \\
\hline
\hline
\multicolumn{1}{|c|}{1} & \multicolumn{1}{c|}{$6.5 \times 10^{-3}$} & \multicolumn{1}{c|}{$\mathcal{O}\left( 10^{-4} \right)$} & \multicolumn{1}{c|}{1.56} & \multicolumn{1}{c|}{2.36} \\
\hline
\multicolumn{1}{|c|}{5} & \multicolumn{1}{c|}{$4.0 \times 10^{-2}$} & \multicolumn{1}{c|}{$\mathcal{O}\left( 10^{-7} \right)$} & \multicolumn{1}{c|}{5.42} & \multicolumn{1}{c|}{6.66} \\
\hline
\multicolumn{1}{|c|}{9} & \multicolumn{1}{c|}{$9.5 \times 10^{-2}$} & \multicolumn{1}{c|}{$\mathcal{O}\left( 10^{-9} \right)$} & \multicolumn{1}{c|}{7.48} & \multicolumn{1}{c|}{8.00} \\
\hline
\multicolumn{1}{|c|}{13} & \multicolumn{1}{c|}{$1.5 \times 10^{-1}$} & \multicolumn{1}{c|}{$\mathcal{O}\left( 10^{-11} \right)$} & \multicolumn{1}{c|}{7.32} & \multicolumn{1}{c|}{8.21} \\
\hline
\multicolumn{1}{|c|}{17} & \multicolumn{1}{c|}{$2.1 \times 10^{-1}$} & \multicolumn{1}{c|}{$\mathcal{O}\left( 10^{-14} \right)$} & \multicolumn{1}{c|}{8.15} & \multicolumn{1}{c|}{8.60} \\
\hline
\end{tabular}
}
\end{center}
\end{table}

\subsection{Singular Burgers equation}\label{section_singular_burgers_equation}
We now consider the inviscid Burgers equation on the domain $(x,t) \in [0,2] \times [0,2]$
\begin{equation}\label{eq_original_burgers}
\setlength{\extrarowheight}{3pt}
 \left\{
 \begin{array}{rcl}
   \dfrac{\partial u}{\partial t} + \dfrac{\partial }{\partial x}\left(\dfrac{u^{2}}{2}\right) & = &  S(x - 1), \\
   u(x,0)                                                                                      & = &  x,  \\
   u(0,t)                                                                                      & = &  0.
 \end{array}
 \right.
\end{equation}
where $S$ is the singular source term in (\ref{eq_original_advection}). The smooth initial data
was chosen such that the corresponding homogeneous problem does not develop shock discontinuities caused by
the nonlinearity of the flux function.

In analogy with what was done in Section \ref{section_singular_advection_equation}, we shall evaluate the accuracy of the spectral
Chebyshev collocation method with $N+1$ Gauss--Lobatto nodes in the solution of (\ref{eq_original_burgers}) on the regions of the
spatial domain $\Omega \eqdef [0,2]$ defined in (\ref{eq_zone_P}) and (\ref{eq_zone_Q}), when $S$ is regularized by
$S^{m,k}_{\varepsilon}$, using the composite Simpson quadrature rule and $N_{p}+1 = 2000$ nonuniform points
\begin{equation*}
\xi_{i} = 1 + \dfrac{3}{10}\sin\left(\pi \left(-\dfrac{1}{2}+\dfrac{i}{N_{p}}\right)\right) \mbox{ for } i = 0,\ldots,N_{p}.
\end{equation*}
We compute the spectral solution at $t = 2$ with a 12th order exponential filter for stabilization, taking $N = 100,200,300,400$.
Due to the absence of analytical solution, the respective error is estimated by comparison with a spectral solution $\tilde{u}_{\varepsilon}(x,2)$
on a finer grid ($N = 500$).

The spectral solution (left) and pointwise error (right) for $m = 13$, $k = 4$ and optimal scaling $\varepsilon = 1.5 \times 10^{-1}$,
are shown in Figure \ref{fig_U_spectral_multiple_particles_burgers_S_13_4}. The spectral solution exhibits an oscillatory behavior around $x = \frac{7}{10}$ (where the first
derivative of $S$ is discontinuous), which decays as $N$ increases. The convergence on $\mathcal{P}_{\varepsilon}$ and
$\mathcal{Q}_{\varepsilon}$ is approximately $\mathcal{O}\left(N^{-5.37}\right)$ and $\mathcal{O}\left(N^{-5.61}\right)$,
respectively. The expected $\mathcal{O}\left(N^{-(k+1)}\right)$ requires at least $m = 13$ vanishing moments. Table \ref{table_m_k_epsi_order_accuracy_burgers} summarizes results
for $m = 5,9,13,17$ and $k = 4$. In contrast to the results for the advection problem in Section \ref{section_singular_advection_equation},
superconvergence as $m$ increases is not observed.

\begin{figure}[h]
\centering
\mbox{ \subfigure{ \includegraphics[width=5.80cm]{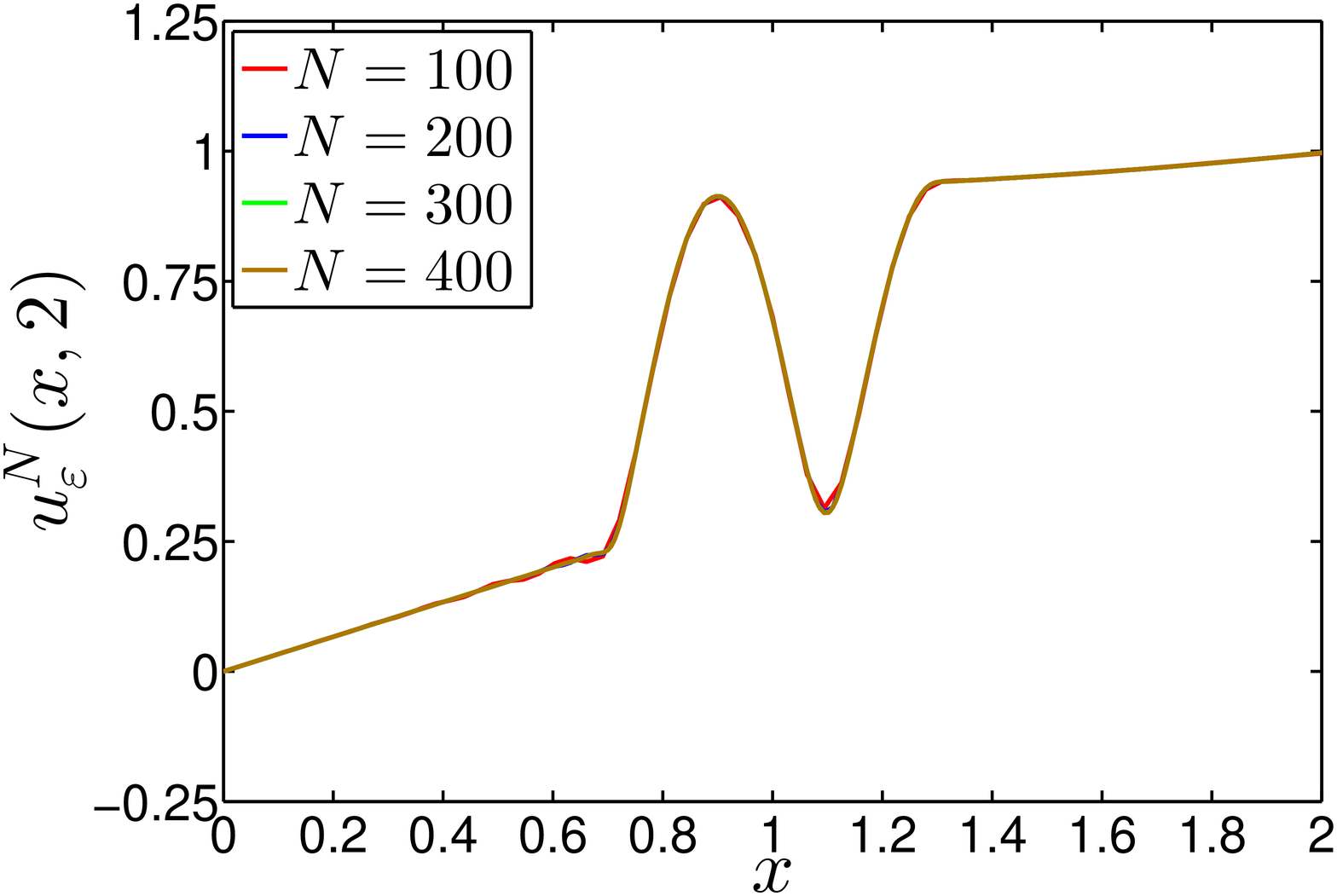} } \quad \subfigure{ \includegraphics[width=5.80cm]{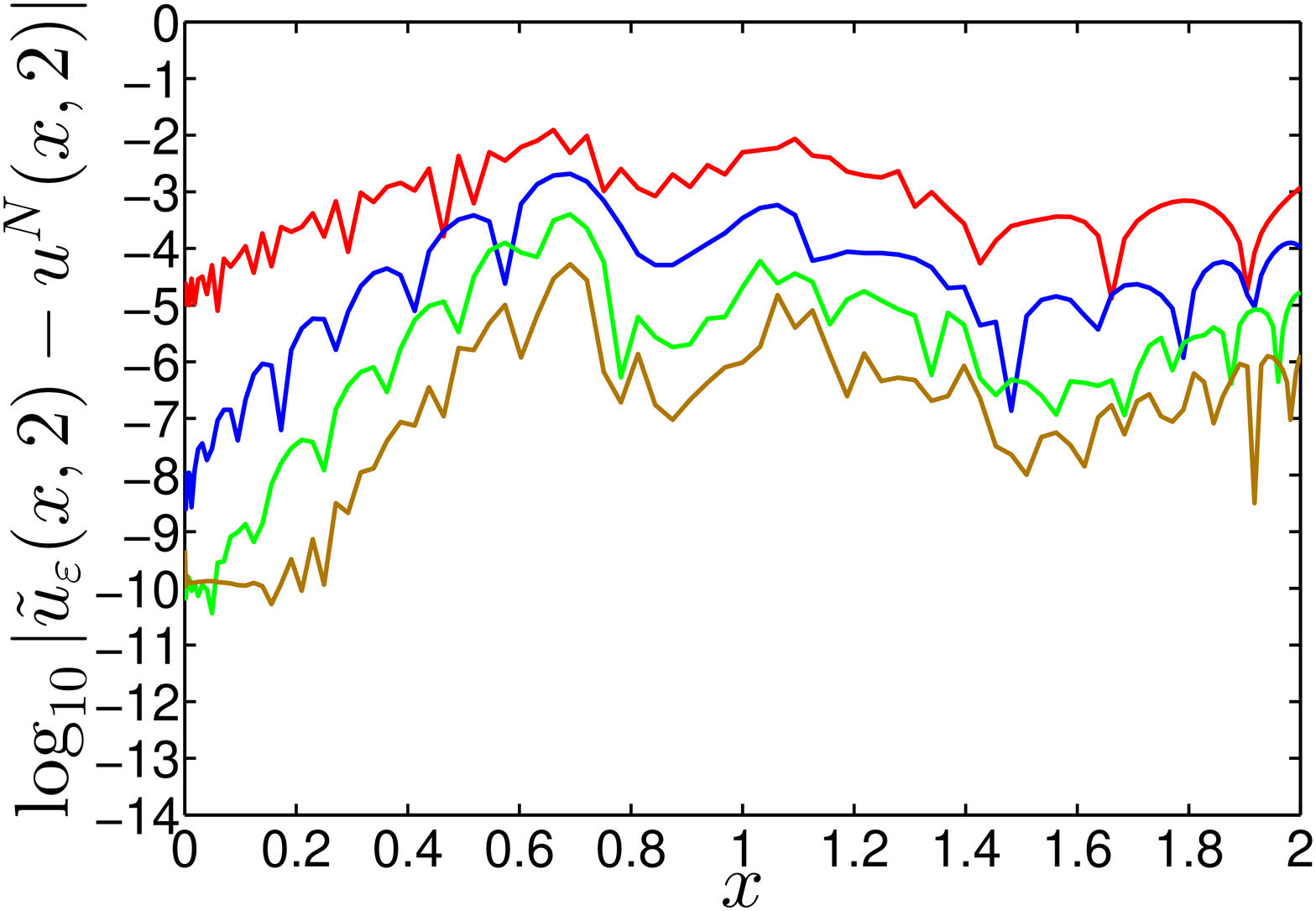} } }
\caption{Spectral solution (left) and pointwise error (right) to the Burgers equation at $t = 2$, using the regularization $S^{m,k}_{\varepsilon}$.}
\label{fig_U_spectral_multiple_particles_burgers_S_13_4}
\end{figure}

\begin{figure}[h]
\centering
\mbox{ \subfigure{ \includegraphics[width=5.80cm]{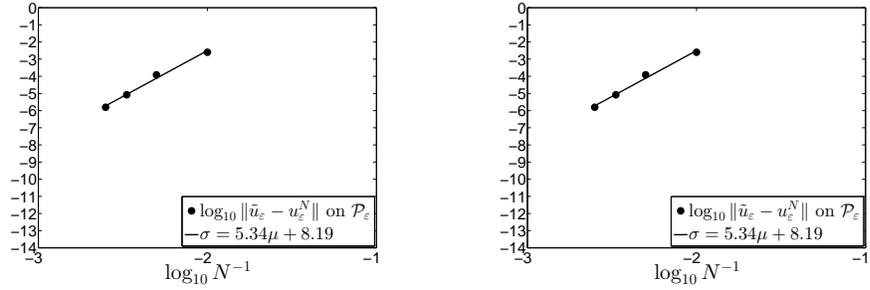} } \quad \subfigure{ \includegraphics[width=5.80cm]{Fig7_left} } }
\caption{Error in $\|\cdot\|_{L^{2}_{w}}$ norm and linear regression to estimate convergence order on $\mathcal{P}_{\varepsilon}$ (left) and $\mathcal{Q}_{\varepsilon}$ (right) to the Burgers equation at $t = 2$.}
\label{fig_conv_error_U_spectral_multiple_particles_burgers_S_13_4}
\end{figure}

\pagebreak

\begin{table}[!h]
\begin{center}
\def\arraystretch{2.0}
\caption{Optimal scaling, pointwise error in the regularization of the source term on $\mathcal{P}_{\varepsilon}$ and convergence order of the spectral solution to the Burgers equation at $t = 2$ on $\mathcal{P}_{\varepsilon} \cup \mathcal{Q}_{\varepsilon}$, when S is replaced by $\tilde{S}^{m,k}$.}
\label{table_m_k_epsi_order_accuracy_burgers}
\footnotesize{
\begin{tabular}{|l|l|l|l|l|}
\hline
\multicolumn{5}{|c|}{$k = 4$} \\
\hline
\multicolumn{1}{|c|}{$m$} & \multicolumn{1}{c|}{$\varepsilon$} & \multicolumn{1}{c|}{$\log_{10} \lvert S(x) - \tilde{S}^{m,k}(x) \rvert$ on $\mathcal{P}_{\varepsilon}$} & \multicolumn{1}{c|}{Conv. order on $\mathcal{P}_{\varepsilon}$} & \multicolumn{1}{c|}{Conv. order on $\mathcal{Q}_{\varepsilon}$} \\
\hline
\multicolumn{1}{|c|}{5} & \multicolumn{1}{c|}{$4.0 \times 10^{-2}$} & \multicolumn{1}{c|}{$\mathcal{O}\left( 10^{-7} \right)$} & \multicolumn{1}{c|}{4.42} & \multicolumn{1}{c|}{3.72} \\
\hline
\multicolumn{1}{|c|}{9} & \multicolumn{1}{c|}{$9.5 \times 10^{-2}$} & \multicolumn{1}{c|}{$\mathcal{O}\left( 10^{-9} \right)$} & \multicolumn{1}{c|}{5.31} & \multicolumn{1}{c|}{3.98} \\
\hline
\multicolumn{1}{|c|}{13} & \multicolumn{1}{c|}{$1.5 \times 10^{-1}$} & \multicolumn{1}{c|}{$\mathcal{O}\left( 10^{-11} \right)$} & \multicolumn{1}{c|}{5.34} & \multicolumn{1}{c|}{5.61} \\
\hline
\multicolumn{1}{|c|}{17} & \multicolumn{1}{c|}{$2.1 \times 10^{-1}$} & \multicolumn{1}{c|}{$\mathcal{O}\left( 10^{-14} \right)$} & \multicolumn{1}{c|}{5.24} & \multicolumn{1}{c|}{5.66} \\
\hline
\end{tabular}
}
\end{center}
\end{table}

\section{Conclusions}\label{conslusions}
We have presented a high-order regularization technique with optimal scaling to approximate singular sources
given by the weighted summation of Dirac-deltas, in the numerical solution of scalar and one-dimensional
advection and Burgers equations with smooth initial data, using the spectral Chebyshev collocation method.

The regularization approximates the source term with high-order of accuracy away from the singularities, according
to the number of vanishing moments, smoothness and the optimal scaling parameter of the approximate Dirac delta in \cite{jp_siam_1}.

Our numerical experiments show that the regularization leads to the expected order of accuracy
in the spectral solution away from the singularities, at a certain number of vanishing moments. In
particular, superconvergence was observed in the linear advection problem as long as the number of vanishing moments
increases.

While the focus of this paper is on  particle methods, the proposed regularization to weighted summation of Dirac-delta functions
using the convolution operator, is generally applicable to a much broader range of problems that involve convolution,
including (but certainly not limited to)  high-order interpolation for image and signal processing, superconvergence
in spectral methods, and higher-order level set methods.

Further work is directed to the implementation of the proposed regularization technique in the numerical
simulation of particle-laden flows with shocks, through the particle-source-in-cell method \cite{jacobs}.


\bibliographystyle{spmpsci}      
\bibliography{biblio}

\begin{thebibliography}{10}
\providecommand{\url}[1]{{#1}}
\providecommand{\urlprefix}{URL }
\expandafter\ifx\csname urlstyle\endcsname\relax
  \providecommand{\doi}[1]{DOI~\discretionary{}{}{}#1}\else
  \providecommand{\doi}{DOI~\discretionary{}{}{}\begingroup
  \urlstyle{rm}\Url}\fi

\bibitem{Abe}
Abe, H., Natsuhiko, S., Itatani, R.: High-order spline interpolations in the
  particle simulation.
\newblock J. Comput. Phys. \textbf{63}, 247--267 (1986)

\bibitem{don_wenoZ_2}
Castro, M., Costa, B., Don, W.S.: High order weighted essentially
  non-oscillatory {W}{E}{N}{O}-{Z} schemes for hyperbolic conservation laws.
\newblock J. Comput. Phys. \textbf{230}(5), 1766--1792 (2011)

\bibitem{spectralWENO}
Costa, B., Don, W.S.: Multi-domain hybrid spectral-{W}{E}{N}{O} methods for
  hyperbolic \linebreak conservation laws.
\newblock J. Comput. Phys. \textbf{224}(2), 970--991 (2007)

\bibitem{psic}
Crowe, C.T., Sharma, M.P., Stock, D.E.: The particle-source in cell
  ({P}{S}{I}-{C}ell) model for gas-droplet flows.
\newblock J. Fluids Eng. \textbf{99}(2), 325--332 (1977)

\bibitem{rabinowitz}
Davis, P.J., Rabinowitz, P.: Methods of numerical integration, second edn.
\newblock Academic Press, Inc. (1984)

\bibitem{folland}
Folland, G.B.: Real analysis: modern techniques and their applications.
\newblock John Wiley \& Sons, Inc. (1984)

\bibitem{guus_friend}
Gelb, A., Cates, D.: Segmentation of images from fourier spectral data.
\newblock Commun. Comput. Phys. \textbf{5}, 326--349 (2009)

\bibitem{Shu}
Gottlieb, S., Shu, C.W.: Total variation diminishing {R}unge--{K}utta schemes.
\newblock Math. Comput. \textbf{67}, 73--85 (1998)

\bibitem{hesthave_sigal_david}
Hesthaven, J.S., Gottlieb, S., Gottlieb, D.: Spectral Methods for
  Time-Dependent Problems, vol.~21.
\newblock Cambridge University Press, Cambridge (2007)

\bibitem{eastwood}
Hockney, R.W., Eastwood, J.W.: Computer simulation using particles.
\newblock Taylor \& Francis Group, New York (1988)

\bibitem{jacobs}
Jacobs, G., Don, W.S.: A high-order {W}{E}{N}{O}-{Z} finite difference scheme
  based \linebreak particle-source-in-cell method for computation of
  particle-laden flows with shocks.
\newblock J. \linebreak Comput. Phys. \textbf{228}, 1365--1379 (2009)

\bibitem{singular_source_1}
Kaufmann, A., Moreau, M., Simonin, O., Helie, J.: Comparison between lagrangian
  and mesoscopic eulerian modelling approaches for inertial particles suspended
  in decaying isotropic turbulence.
\newblock J. Comput. Phys. \textbf{227}(13), 6448--6472 (2008)

\bibitem{image_interpolation_5}
Malgouyres, F., Guichard, F.: Edge direction preserving image zooming: A
  mathematical and numerical analysis.
\newblock SIAM J. Numer. Anal. \textbf{39}(1), 1--37 (2001)

\bibitem{image_interpolation_1}
Meijering, E.H.W., Niessen, W.J., Viergever, M.A.: Quantitative evaluation of
  convolution based methods for medical image interpolation.
\newblock Med. Image Anal. \textbf{5}(2), 111--126 (2001)

\bibitem{Monaghan}
Monaghan, J.J.: Shock simulation by the particle method {S}{P}{H}.
\newblock J. Comput. Phys. \textbf{52}, 374--389 (1983)

\bibitem{Monaghan_2}
Monaghan, J.J.: Extrapolating {B} splines for interpolation.
\newblock J. Comput. Phys. \textbf{60}, 253--262 (1985)

\bibitem{sph_monaghan_2005}
Monaghan, J.J.: Smoothed particle hydrodynamics.
\newblock Rep. Prog. Phys. \textbf{68}, 1703--1759 (2005)

\bibitem{image_interpolation_4}
Panda, R., Chatterji, B.N.: Least squares generalized {B}-spline signal and
  image processing.
\newblock Signal Process. \textbf{81}, 2005--2017 (2001)

\bibitem{ryan_siac}
Ryan, J., Shu, C.W.: On a one-sided post-processing technique for the
  discontinuous {G}alerkin methods.
\newblock Methods Appl. Anal. \textbf{10}(2), 295--308 (2003)

\bibitem{image_interpolation_3}
Schaum, A.: Theory and design of local interpolators.
\newblock CVGIP: Graph. Model. Im. Proc. \textbf{55}(6), 464--481 (1993)

\bibitem{shu_eno_weno}
Shu, C.W.: Essentially non-oscillatory and weighted essentially non-oscillatory
  schemes for hyperbolic conservation laws.
\newblock Lect. Notes Math. \textbf{1697}, 325--432 (1998)

\bibitem{jp_siam_1}
Suarez, J.P., Jacobs, G., Don, W.S.: A high-order {D}irac-delta regularization
  with optimal scaling in the spectral solution of one-dimensional singular
  hyperbolic conservation laws.
\newblock SIAM J. Sci. Comput. \textbf{36}(4), A1831--A1849 (2014)

\bibitem{image_interpolation_2}
Th\'evenaz, P., Blu, T., Unser, M.: Interpolation revisited.
\newblock IEEE T. Med. Imaging \textbf{19}(7), 739--758 (2000)

\bibitem{Tornberg}
Tornberg, A.K.: Multi-dimensional quadrature of singular and discontinuous
  functions.
\newblock BIT \textbf{42}(3), 644--669 (2002)

\bibitem{shu_2013}
Yang, Y., Shu, C.W.: Discontinuous {G}alerkin method for hyperbolic equations
  involving $\delta$-singularities: negative-order norm error estimates and
  applications.
\newblock Numer. Math. \textbf{124}, 753--781 (2013)

\end{thebibliography}


\end{document}